\documentclass[11pt]{article}
\usepackage{amsfonts}
\usepackage{dsfont}
\usepackage{amsfonts}
\usepackage{bbm}
\usepackage{mathrsfs}
\usepackage{color}
\usepackage{hyperref}
\usepackage{mathrsfs}
\usepackage{amsmath}
\usepackage{enumitem}
\allowdisplaybreaks[4]
\usepackage{amsfonts}
\usepackage{amsthm}
\usepackage{amssymb, amsmath, cite}
\usepackage{bm}
\usepackage{amsthm}
\usepackage{verbatim}
\usepackage{graphicx}
\usepackage{color}
\usepackage{amsfonts}
\usepackage{dsfont}
\usepackage{amsfonts}
\usepackage{bbm}
\usepackage{mathrsfs}
\usepackage{color}
\usepackage{graphicx,amsmath,amsfonts,amssymb,amscd,amsthm}
\usepackage{hyperref}
\usepackage[numbers,sort&compress]{natbib}
\usepackage{enumitem}
\setenumerate[1]{itemsep=1pt,partopsep=0pt,parsep=\parskip,topsep=1pt}
\setitemize[1]{itemsep=1pt,partopsep=0pt,parsep=\parskip,topsep=1pt}

\newtheorem{theorem}{Theorem}[section]
\newtheorem{lemma}{Lemma}[section]

\newtheorem{remark}{Remark}[section]
\newtheorem{definition}{Definition}[section]

\usepackage{graphicx,amsmath,amsfonts,amssymb,amscd,amsthm}
\usepackage[numbers,sort&compress]{natbib}
\usepackage{array}
\usepackage{booktabs}
\usepackage{multirow}
\usepackage{makecell}
\usepackage{abstract}
\hypersetup{colorlinks,linkcolor=blue,citecolor=blue}
\newtheoremstyle{kai}
{3pt} {3pt} {} {} {\bfseries} {.} {.5em} {}
\def\EquationsBySection{\def\theequation
{\thesection.\arabic{equation}}
\@addtoreset{equation}{section}}

\newcommand\old[1]{}

 \newcommand{\Bp}{\begin{proof}}
 \newcommand{\Ep}{\end{proof}}
\renewcommand{\theequation}{\thesection.\arabic{equation}}
  \newcommand{\red}{\textcolor[rgb]{1.00,0.00,0.00}}
 
\DeclareMathOperator{\dive}{div}
\numberwithin{equation}{section}
\addcontentsline{toc}{section}{Acknowledgement}
\begin{document}
\title{\bf A Counterexample for the Principal Eigenvalue of An Elliptic Operator with Large Advection }
\author{{\sc Xueli Bai} \thanks {School of Mathematics and Statistics, Northwestern Polytechnical University, 
Xi'an 710072, China;  xlbai2015@nwpu.edu.cn}
\and {\sc Xin Xu}
\thanks{Chern Institute of Mathematics, Nankai University, Tianjin 300071, China; 9820210107@nankai.edu.cn
 }\and
 {\sc Kexin Zhang}\thanks {Chern Institute of Mathematics, Nankai University, Tianjin 300071, China; kxzmath@163.com}
 \and {\sc Maolin Zhou}\thanks{Chern Institute of Mathematics and LPMC, Nankai University, Tianjin 300071, China; zhouml6@nankai.edu.cn}}
\date{}
\maketitle
\begin{abstract}
 There are numerous studies focusing on the convergence of the principal eigenvalue $\lambda(s)$ as $s\to+\infty$ corresponding to the elliptic eigenvalue problem
 \begin{align*}
  -\Delta\varphi(x)-2s\mathbf{v}\cdot\nabla\varphi(x)+c(x)\varphi(x)=\lambda(s)\varphi(x),\quad x\in \Omega,
 \end{align*}
  where $\Omega$ is a bounded domain and  the advection term $\mathbf{v}$ under some certain restrictions.
  In this paper, we construct an infinitely oscillating gradient advection term $\mathbf{v}=\nabla m(x)\in C^1(\Omega)$ such that the principal eigenvalue $\lambda(s)$ does not converge as $s\to+\infty$.
As far as we know, this is the first result that guarantee the non-convergence of the principal eigenvalue.
\textbf{AMS (2010) Subject Classification}: {35P15, 35P20, 34C25.}
\end{abstract}

\section{Introduction}
The spectral theory of differential operators finds extensive applications in many mathematical fields, such as differential equations, differential  geometry and fluid mechanics. The principal eigenvalue associated with this linear eigenvalue problem plays a crucial role in various aspects \cite{DV1975,DV1976,nadin2011,PW1966,MNP2022,BNV1994,nadin2009}, including speed estimation of traveling fronts \cite{BH2002,BHN2005} and rearrangement inequalities of eigenvalues \cite{HNR2011}, among others.
One particular area of interest lies in studying how different phenomena, such as advection, affect the speed of fronts. This involves exploring singular limits within the eigenvalue problem.
Significant research efforts have been dedicated to investigating the asymptotic behavior of the principal eigenvalue for second-order elliptic operators.

There are many researches focusing  on the eigenvalue problem in a bounded domain $\Omega\subset \mathbb{R}^d$, with $ d\ge1$:
  \begin{equation}\label{BF}
  \left\{
\begin{array}{l}
  -\Delta\varphi(x)-2s\mathbf{v}\cdot\nabla\varphi(x)+c(x)\varphi(x)=\lambda(s)\varphi(x),\quad x\in \Omega,\\
  \nabla\varphi\cdot\mathbf{n}=0,\quad x\in\partial\Omega,
  \end{array}\right.
  \end{equation}
  where   $\mathbf{v}$ is a vector field, $c(x)\in C(\bar{\Omega})$ and $\mathbf{n}$ is the unit outward normal on $\partial\Omega$.  It is well known that there exists a unique, positive principal eigenfunction $\varphi(s;x)$  (under normalization),   corresponding to the principal eigenvalue $\lambda(s)$, which is real and simple.\par

When $\mathbf{v}$ is a divergence free  vector field  ($\dive{\mathbf{v}}=0$), satisfying $\mathbf{v}\cdot \mathbf{n}=0$ on $\partial \Omega$, Berestycki, Hamel and Nadirashvili in \cite{Bcmp}   
   discovered that $\lambda(s)$ converges to the minimum value of a specific functional as $s\to+\infty$.
If $\mathbf{v}$ is a gradient vector field, that is $\mathbf{v}=\nabla m(x)$ for some $m(x)\in C^2(\bar{\Omega})$,
Chen and Lou in \cite{CL} 
 established that for the non-degenerate $m(x)$, 
  the principal eigenvalue $\lambda(s)$ satisfies
  \begin{align}\label{clr}
  \lim_{s\to+\infty}\lambda(s)=\min_{\mathcal{M}}c(x),
  \end{align}
  where $\mathcal{M}$ is the  set of the local maximum points of $m(x)$.
  In one dimensional space, if   $m(x)$ vanishes on the interval $[a,b]$ and changes sign finite times, Peng and Zhou \cite{PZ}  investigated that
$\lambda(s)$  converges. 
If $\mathbf{v}$ is a general vector field, in the two-dimensional space, in the work referred to
\cite{LLPZ} by Liu, Lou and Zhou, it is established that if  the corresponding dynamical system
$
\dot{\mathbf{X}}=\mathbf{v}(\mathbf{X})
$
possesses  finite fixed points, degenerate regions, closed orbits, stable limit cycles and homoclinic orbits connected by saddles,
the principal eigenvalue $\lambda(s)$ converges as $s\to+\infty$.
\par

Regarding other related work, for example the large advection problem for the linear time-periodic parabolic operator and
the small diffusion problem for the linear elliptic problem or the linear time-periodic parabolic operator, that is multiply a small parameter $\varepsilon$ in front of the  diffusion term  to investigate the limiting behavior of $\lambda(\varepsilon)$ as $\varepsilon\to 0$. There are also many convergence results, and we may refer to \cite{FA7273,DEF73,LLPZd,LLPZx}.

\par
Based on the existing results, it is known that for the advection term  under certain restrictions, such as finite oscillating, the principal eigenvalue $\lambda(s)$ of \eqref{BF} converges as $s\to+\infty$.
Whereas, for a general advection term, the convergence of $\lambda(s)$  is still unknown.
It is worth mentioning that Berestycki, Hamel and Nadirashvili in their remarkable work \cite{Bcmp} proved  that under the Neumann boundary condition, boundedness means the existence of limit for a divergence free vector field and asked an open problem that for the equation in \eqref{BF},  with Robin type boundary conditions and a divergence free $\mathbf{v}$, does boundedness mean convergence? Investigating the convergence or non-convergence of $\lambda(s)$ remains a highly interesting topic of research.
\par

In this paper, we consider the radial symmetric problem of  \eqref{BF} in the domain $\Omega=B_{1}(0)$, with the gradient vector field $\mathbf{v}(x)=\nabla m(x)$. Assume that $m(x)$ and $c(x)$ are radial symmetric and we write  $m(r)=m(|x|)$ and $c(r)=c(|x|)$ where
 $r=|x|$, then equation \eqref{BF} converts into
 \begin{equation}
\left\{
\begin{array}{l}
-\varphi''(r)-\frac{d-1}{r}\varphi'(r)-2sm'(r)\varphi'(r)+c(r)\varphi(r)=\lambda(s,m(r))\varphi(r),\quad 0<r<1,\\
\varphi'(0)=\varphi'(1)=0,
\end{array}\right.\label{eq}
\end{equation}
where $\lambda(s,m(r))$ is the principal eigenvalue associated with $m(r)$ and $\varphi(s;r)$ is the positive principal eigenfunction with the normalization $\int_0^1r^{d-1}e^{2sm(x)}(|\varphi'^2(s;x)+c(x)\varphi^2(s;x)|)dx=1$.
We construct an infinitely oscillating  function $ m(r)\in C^1([0,1])$ 
such that the corresponding principal eigenvalue $\lambda(s,m(r))$ does not converge as $s\to+\infty$. To our knowledge, this is the first counterexample  to reveal the non-existence of the limit of principal eigenvalue $\lambda(s,m(r))$. This result also implies that  for highly complex, infinitely oscillating advection term, determining the convergence behavior of $\lambda(s,m(r))$ is fundamentally impossible unless appropriate conditions are imposed on the advection term, regardless of whether for $\lambda(s)$ of equation \eqref{BF} or $\lambda(s,m(r))$ of equation \eqref{eq}.

To state our main result, we first introduce  a definition and  two hypotheses. In the following, we consider on the interval $[\frac{1}{3},\frac{2}{3}]$. It is also feasible for general interval $[a,b]$.
\begin{definition}\label{d2}
Define $\lambda^{ij}$ to be the principal eigenvalue of the elliptic principal eigenvalue problem
\begin{align*}
-\varphi''(r)-\frac{d-1}{r}\varphi'(r)+c(r)\varphi(r)=\lambda\varphi(r),\quad r\in\left(\frac{1}{3},\frac{2}{3}\right)
\end{align*}
subject to the Neumann (Dirichlet, respectively) boundary condition at the left boundary point $x=\frac{1}{3}$ if $i=\mathcal{N}$ (if $i=\mathcal{D}$, respectively) and
subject to the Neumann (Dirichlet, respectively) boundary condition at the right boundary point $x=\frac{2}{3}$ if $i=\mathcal{N}$ (if $j=\mathcal{D}$, respectively).
\end{definition}
\noindent\textbf{(H1)}
 $c(r)>\lambda^{\mathcal{DD}} \text{ as } x\in [0,\frac{1}{3}]\cup[\frac{2}{3},1]$.
 \\
 \textbf{(H2)}
 $c(r)>\lambda^{\mathcal{NN}} \text{ as } x\in [0,\frac{1}{3}]\cup[\frac{2}{3},1]$.
 \par
 It is noted that if (H1) holds true then (H2) holds true by the maximum principle \cite{GT}.
\begin{theorem}\label{robin}
Suppose that  $c(r)$ satisfies (H1), then there exists an infinitely oscillating $m(r)\in C^1([0,1])$ such that the principal eigenvalue $\lambda(s,m(r))$  in \eqref{eq} does not converge as $s\to+\infty$.
\end{theorem}


\par

The sketch of the proof of Theorem \ref{robin} is as follows: we first prove that the limits of the principal eigenvalue $\lambda(s,m(r))$ corresponding to two types of infinitely oscillating $m(r)$   tend to $\lambda^{\mathcal{DD}}$ and $\lambda^{\mathcal{NN}}$, respectively (see Section 3 and Section 4).
In Section 5, we then established the continuity dependence of $\lambda(s,m(r))$ on   $m(r)$.
This is crucial as it enables us to finely adjust $m(r)$.
Specifically, we  select one type $m_1(r)$ whose principal eigenvalue $\lambda(s,m_1(r))$ converges to $\lambda_1^{\mathcal{DD}}$.
For some $s=s_1$, by making a small adjustment to obtain the other type $m_2(r)$, $\lambda(s,m_2(r))$ remains  closed to $\lambda^{\mathcal{DD}}$  while converging to $\lambda^{\mathcal{NN}}$ for large $s$. Based on this strategy, we can select two sequences $\{s_{2n}\}_{n=1}^{+\infty}$ and $\{s_{2n+1}\}_{n=1}^{+\infty}$ of $s\to+\infty$ and  construct two type function sequences $\{ m_{2n}(r) \}_{n=1}^{+\infty}$ and $\{m_{2n+1}(r)\}_{n=1}^{+\infty}$
whose limits of principal eigenvalues converge to two different constants respectively.
Finally, we obtain a counterexample  $\hat{m}(r)=\lim_{n\to +\infty}m_n(x)$, where the corresponding principal eigenvalue does not converge. This completes the proof of Theorem \ref{robin}.

\section{Preliminaries}\label{S:pre}

Note that the principal eigenvalue $\lambda(s,m(r))$ of \eqref{eq} can be characterized by
\begin{eqnarray}
\lambda(s,m(r))
&=&\min_{\int_0^1 r^{d-1}e^{2sm(r)}\varphi^2(r)dr=1 }\int_0^1 r^{d-1}e^{2sm(r)}\left(|\varphi'(r)|^2+c(r)|\varphi(r)|^2\right)dr.\label{voe}
\end{eqnarray}
Denote $c_*\triangleq\min_{r\in[0,1]}c(r)$ and $c^*\triangleq\max_{r\in[0,1]}c(r)$. From the variational formulation \eqref{voe}, we can establish the following upper and lower bounds of the principal eigenvalue,
\begin{equation}
c_*\le \lambda(s,m(r))\le  c^*\quad \forall\ s>0.\label{ubd}
\end{equation}
Without loss of generality, we assume that $c(r)>0$ on $[0,1]$ in the rest of this paper. Otherwise, we can modify equation \eqref{eq} by replacing  $(\lambda(s,m(r)),c(r))$ with
$
(\lambda(s,m(r))+\max_{[0,1]}|c(r)|,c(r)+\max_{[0,1]}|c(r)|).
$
This ensures that $0<c_*\leq c^*$. \par

We denote $\varphi(s;r)$ as the principal eigenfunction corresponding to $\lambda(s,m(r))$ of \eqref{eq} for any fixed $s$, satisfying the normalization condition $\int_0^1r^{d-1}e^{2sm(r)}|\varphi(s;r)|^2dr=1$ for each $s>0$.
In light of \eqref{voe}, the principal eigenvalue $\lambda(s,m(r))$ can be equivalently characterized by
\begin{eqnarray}
\lambda(s,m(r))
&=&\min_{\int_0^1 w(s;r)^2dr=1}\int_0^1 \left|w'(s;r)-sm'(r)w(s;r)-\frac{d-1}{2r}w(s;r)\right|^2+c(r)|w(s;r)|^2dr.\label{voa}
\end{eqnarray}
 by taking the transformation $w(s;r)=e^{sm(r)}r^{\frac{d-1}{2}}\varphi(s;r)$.
Recall that $\{w^2(s;\cdot)\}_{s>0}$ is weakly compact, which means that there exists a sequence $\{s_i\}_{i=1}^{+\infty}$ with $s_i\to+\infty$ as $i\to+\infty$ such that
\begin{equation}\label{weakc}
\lim_{i\to+\infty}\int_0^1w^2(s_i;r)\zeta(r)dr=\int_{[0,1]}\zeta(r)\mu(dr),\quad \forall \zeta\in C([0,1]),
\end{equation}
where $\mu$ is a certain probability measure (see \cite{CL}).

Now, we present a lemma to demonstrate the convergence property of $\varphi(s;r)$ on the interval $(\frac{1}{3},\frac{2}{3})$ as $s\to+\infty$. And for convenient use, we denote by $C$ for any positive constant which may change from line to line in the  entire paper.
\begin{lemma}\label{converge}
Suppose that $m(r)=0$ on the interval $[\frac{1}{3},\frac{2}{3}]$, then
there exist a sequence $\{s_i\}_{i=1}^{+\infty}$ with $s_i\to+\infty$ as $i\to+\infty$ and a function $\varphi_*(r)$ such that $\varphi(s_i;r)\to \varphi_*(r)$ in $C^1_{loc}([\frac{1}{3},\frac{2}{3}])$ as $i\to+\infty$. Moreover, $\varphi_*(r)$ satisfies
\begin{equation}
-\varphi_*''(r)-\frac{d-1}{r}\varphi_*'(r)+c(r)\varphi_*(r)=\lambda_*\varphi_*(r),\quad r\in \left(\frac{1}{3},\frac{2}{3}\right),\label{v*}
\end{equation}
in the weak sense, where $\lambda_*=\liminf_{s\to+\infty}\lambda(s,m(r))$.
\end{lemma}

\begin{proof}
First we select a sequence $\{s_i\}_{i=1}^{+\infty}$ such that \eqref{weakc} and $\lim_{i\to+\infty}\lambda_1^{\mathcal{N}}(s_i,m(r))=\lambda_*$ hold true.
Since $m(r)=0$ for $r\in[\frac{1}{3},\frac{2}{3}]$, by \eqref{ubd} and the variational characterization of $\lambda_1^{\mathcal{N}}(s,m(r))$, we have
\begin{align*}
c^*&\ge \lambda(s_i,m(r))\ge \left(\frac{1}{3}\right)^{d-1}\int_{\frac{1}{3}}^{\frac{2}{3}}  |\varphi'(s_i;r)|^2+c(r)|\varphi(s_i;r)|^2dr
\ge C\int_{\frac{1}{3}}^{\frac{2}{3}} |\varphi'(s_i;r)|^2+|\varphi(s_i;r)|^2dr,
\end{align*}
where $C=(\frac{1}{3})^{d-1}\min\{1,c_*\}>0$.
According to the compact embedding theorem, there exist a subsequence of $\{s_{i}\}_{i=1}^{+\infty}$, which is denoted by itself, such that
\begin{equation}
\varphi(s_{i};r)\to \varphi_*(r), \quad\text{ in }\ C\left(\left[\frac{1}{3},\frac{2}{3}\right]\right),\quad i\to+\infty,\label{m}
\end{equation}
and a constant $C>0$, independent of $i$, such that
\begin{align}\label{wb}
|\varphi(s_{i};r)|\le C,\quad r\in\left[\frac{1}{3},\frac{2}{3}\right].
\end{align}
Moreover,  from equation \eqref{eq}, we have
\begin{equation}\label{d}
-\varphi''(s_{i};r)-\frac{d-1}{r}\varphi'(s_i;r)+c(r)\varphi(s_{i};r)=\lambda(s_{i},m(r))\varphi(s_{i};r),\quad r\in\left(\frac{1}{3},\frac{2}{3}\right).
\end{equation}
So by \eqref{wb}, the boundedness of terms in \eqref{d} and the compactness, we could further extract a subsequence of $\{s_{i}\}_{i=1}^{+\infty}$, denoted by itself, such that $\varphi(s_{i};r)\to\varphi_*(r)$ in $C^1_{loc}((\frac{1}{3},\frac{2}{3}))$ and  $\varphi_*(r)$ satisfies
\begin{equation*}
-\varphi_*''(r)-\frac{d-1}{r}\varphi_*'(r)+c(r)\varphi_*(r)=\lambda_*\varphi_*(r),
\end{equation*}
in the weak sense.
\end{proof}



\medskip
As the assumption in Lemma \ref{converge}, throughout this paper, we always assume that the function $m(r)$ in \eqref{eq} vanishes on the interval $[\frac{1}{3},\frac{2}{3}]$ and $m(r)=m(1-r)$ on $[0,1]$. Here we give two types of test function $\tilde{m}(r)$ and $\bar{m}(r)$ to describe $m(r)$ in the following.
\begin{definition}\label{DM}
(1). Let $\delta\in(0,\frac{1}{3})$ and  $\alpha,\ \beta $  be constants satisfying the conditions $\frac{1}{6}<\alpha< \beta<1$ and $\sum_{i=1}^{+\infty}(\alpha^i+\beta^i)=\frac{1}{3}-\delta$. We consider two sequences of points $\{x_n\}_{n=0}^{+\infty}$ and $\{y_n\}_{n=0}^{+\infty}$ on $[\delta,\frac{1}{3})$ with $x_0=\delta$, $y_0=\delta+\alpha$, $x_n=\delta+\sum_{i=1}^{n}(\alpha^i+\beta^i)$ and $y_n=\delta+\sum_{i=1}^{n}(\alpha^i+\beta^i)+\alpha^{n+1}$ for $n\ge1$. Then we define $\tilde{m}(\alpha,\beta;r)$ for $r\in[\delta,\frac{1}{3})$ as follows
\begin{equation}\label{tilm}
\tilde{m}(\alpha,\beta;r) = \left\{ \begin{array}{lll}
-\left(\frac{1}{6}\right)^n, &r\in [x_n, y_n),\,\, n\in \mathbb{N},\\
2 \left(\frac{1}{6}\right)^n, &r\in [ y_n,x_{n+1}),\,\,n\in \mathbb{N}.
 \end{array} \right.
\end{equation}
For $r\in(\frac{2}{3},1-\delta]$, we define $\tilde{m}(\alpha,\beta;r)=\tilde{m}(\alpha,\beta;1-r)$.
\par
(2). Let $\delta\in(0,\frac{1}{3})$ and  $\alpha$ be constant satisfying $\frac{1}{6}<\alpha<1$ and $2\sum_{i=1}^{+\infty} \alpha^i=\frac{1}{3}-\delta$.  We consider two sequences of points $\{X_n\}_{n=1}^{+\infty}$ and $\{Y_n\}_{n=0}^{+\infty}$ on $[\delta,\frac{1}{3})$ with
$Y_0=\delta$, and
$X_n=\delta+2\sum_{i=1}^n \alpha^i-\alpha^n$,
$Y_n=\delta+2\sum_{i=1}^n \alpha^i$ for $n\ge1$.
Then we define $\bar{m}(\alpha;r)$ for $r\in [0,\frac{1}{3})$ as follows
\begin{equation}\label{barm}
\bar{m}(\alpha;r) = \left\{ \begin{array}{lll}
\left(\frac{1}{6}\right)^{n},& r\in[Y_{n-1},X_n],\,n\in\mathbb{N}^+,\\
-2 \left(\frac{1}{6}\right)^{n},&r\in[X_n,Y_n],\, n\in\mathbb{N}^+.
 \end{array} \right.
\end{equation}
For $(\frac{2}{3},1-\delta]$, we define $\bar{m}(\alpha;r)=\bar{m}(\alpha;1-r)$.
\end{definition}
Since the test functions $\tilde{m}(r)$ and $\bar{m}(r)$ are employed to describe the later chosen function $m(r)$, we only define them on some essential intervals. And based on the definition of functions $\tilde{m}(\alpha,\beta;r)$ and $\bar{m}(\alpha;r)$, we construct the following two sets.
\begin{definition}\label{e1}

(1). (The set $S_{\mathcal{DD}}$)
We say $m(r)\in S_{\mathcal{DD}}$ if $m(r)\in C^1([0,1])$, vanishes on $[\frac{1}{3},\frac{2}{3}]$, and there exist constants $\alpha$ and $\beta$ satisfying $\frac{1}{6}<\alpha< \beta<1$, such that for any $r\in [\delta,\frac{1}{3})\cup(\frac{2}{3},1-\delta]$, $m(r)\ge \tilde{m}(\alpha,\beta;r)$.\par
(2). (The set $S_{\mathcal{NN}}$)
We say $m(r)\in S_{\mathcal{NN}}$ if $m(r)\in C^1([0,1])$, vanishes on $[\frac{1}{3},\frac{2}{3}]$, and there exist constant $\alpha$ satisfying $\frac{1}{6}<\alpha$, such that for $r\in[\delta,\frac{1}{3})\cup(\frac{2}{3},1-\delta]$, $m(r)\leq  \bar{m}(\alpha;r)$.
\end{definition}

\section{Limit of the principal eigenvalue  for $m(r)\in S_{\mathcal{DD}}$}

\begin{theorem}\label{dirichlet}
 Suppose that (H1) holds and $m(r)\in S_{\mathcal{DD}}$,  then  the principal eigenvalue $\lambda(s,m(r))$ of \eqref{eq} satisfies
 $$\lim_{s\to+\infty}\lambda(s,m(r))=\lambda^{\mathcal{DD}}$$
\end{theorem}
We  prove it by estimating the upper and lower limits of  $\lambda(s,m(r))$, respectively. We write $\lambda(s)$ for some fixed $m(r)$ for convenience in this and next section.

\begin{lemma}[Upper bound]\label{upd}
Under conditions in Theorem \ref{dirichlet}, we have
 $$\limsup_{s\to+\infty}\lambda(s)\le \lambda^{\mathcal{DD}}.$$
\end{lemma}
\begin{proof}
We construct a test function $\varphi_1(r)$ with $\varphi_1(r)=0$ on $[0,\frac{1}{3})\cup(\frac{2}{3},1]$ and $\varphi_1(r)=\varphi_{\mathcal{DD}}(r)$ on $[\frac{1}{3},\frac{2}{3}]$, where $\varphi_{\mathcal{DD}}(r)$ is the principal eigenfunction to the equation
 \begin{eqnarray*}
 \left\{
 \begin{array}{l}
 -\varphi''(r)-\frac{d-1}{r}\varphi'(r)+c(r)\varphi(r)=\lambda\varphi(r),\quad \frac{1}{3}<r<\frac{2}{3},\\
 \varphi(\frac{1}{3})=\varphi(\frac{2}{3})=0.
 \end{array}\right.
\end{eqnarray*}
 By the variational characterization of $\lambda(s)$ and $\lambda^{\mathcal{DD}}$, noting that $m(r)=0$ on $[\frac{1}{3},\frac{2}{3}]$, we have
 \begin{eqnarray*}
 \lambda(s)\le\frac{\int_0^1 r^{d-1} e^{2sm(r)}\left(|\varphi_1'|^2+c(r)|\varphi_1|^2\right)dr}{\int_0^1 r^{d-1} e^{2sm(r)}|\varphi_1|^2dr}=
 \frac{\int_{\frac{1}{3}}^{\frac{2}{3}} r^{d-1} \left(|\varphi_1'|^2+c(r)|\varphi_1|^2\right)dr}{\int_{\frac{1}{3}}^{\frac{2}{3}} r^{d-1} |\varphi_1|^2dr}=
 \lambda^{\mathcal{DD}}.
 \end{eqnarray*}
Taking the upper limit of $s$ on both sides, we finish the proof.
\end{proof}

Before giving the lower bound, we recall that in Lemma \ref{converge}, there exists $\varphi_*(r)\in C_{loc}^1([\frac{1}{3},\frac{2}{3}])$ that satisfies
\begin{equation}\label{r}
-\varphi_*''(r)-\frac{d-1}{r}\varphi_*'(r)+c(r)\varphi_*(r)=\lambda_*\varphi_*(r)\quad r\in(\frac{1}{3},\frac{2}{3}),
\end{equation}
in weak sense. To give an explicit $\varphi_*(r)$, we
are needed to determine the boundary conditions of $\varphi_*(r)$ in the following lemma, that is, $\varphi_*(\frac{1}{3})=\varphi_*(\frac{2}{3})=0 $.

\begin{lemma}\label{id}
If $m(r)\in S_{\mathcal{DD}}$, then the function $\varphi_*(r)$ derived from Lemma \ref{converge}, satisfies \eqref{v*} in the weak sense, subject to Dirichlet boundary condition, that is,
\begin{eqnarray}
 \left\{
 \begin{array}{l}
 -\varphi_*''(r)-\frac{d-1}{r}\varphi_*'(r)+c(r)\varphi_*(r)=\lambda_*\varphi_*(r),\quad \frac{1}{3}<r<\frac{2}{3},\\
 \varphi_*(\frac{1}{3})=\varphi_*(\frac{2}{3})=0.
 \end{array}\right.
\end{eqnarray}
\end{lemma}
\begin{proof}
It suffices to prove $\varphi_*(\frac{1}{3})=\varphi_*(\frac{2}{3})=0$, based on Lemma \ref{converge}.
If $\varphi_*(\frac{1}{3})\ne 0$, due to the non-negativity of $\varphi_*(r)$ on $(\frac{1}{3},\frac{2}{3})$, we have $\varphi_*(\frac{1}{3})>0$. Then we can define
$$
N(s_{i})=\inf\left\{l\in \mathbb{N}\bigg|\varphi(s_{i};r)>\frac{1}{2}\varphi_*(\frac{1}{3}),\ r\in[y_n,x_{n+1}],\ \text{for all }\ n\ge l\right\},
$$
where the sequence  $\{s_i \}_{i=1}^{+\infty}$ is the one obtained in Lemma \ref{converge}, and  $y_n$ and $x_{n+1}$ are given in Definition \ref{e1}.
It is easy to see $N(s_{i})$ is well defined. Actually, we have
\begin{equation}\label{N}
N(s_{i})\to+\infty,\ \text{as } i\to+\infty.
\end{equation}
Recall that $m(r)\in S_{\mathcal{DD}}$ implies that there exists a function $\tilde{m}(\alpha,\beta;r)$ such that $m(r)\ge\tilde{m}(r)$ on $[\delta,\frac{1}{3})$. If $N(s_{i})\le N^*$ for some positive constant $N^*$ for all large $i$, we obtain
\begin{eqnarray*}
 \int_{y_{N(s_{i})}}^{x_{N(s_{i})+1}}r^{d-1}e^{2s_{i}m} (|\varphi'|^2+c|\varphi|^2)dr
\ge \delta^{d-1}\int_{y_{N(s_{i})}}^{x_{N(s_{i})+1}}e^{2s_{i}\tilde{m}} c|\varphi|^2dr
\ge
 \frac{1}{4}c_*\delta^{d-1}\varphi_*^2(\frac{1}{3})\beta^{N^*+1}e^{4 s_{i} (\frac{1}{6})^{N^*}},
\end{eqnarray*}
which tends to infinity as $i\to+\infty$. 
However, it is impossible according to the following inequality,
\begin{equation}\label{key}
c^*\ge \int_\delta^{\frac{1}{3}} r^{d-1}e^{2sm}(|\varphi'|^2+c|\varphi|^2)dr\ge  \int_{y_{N(s_{i})}}^{x_{N(s_{i})+1}}r^{d-1}e^{2s_{i}m} (|\varphi'|^2+c|\varphi|^2)dr,\quad\forall\ s>0.
\end{equation}
So \eqref{N} holds true.\par

\par
From inequality \eqref{key}, we have
\begin{eqnarray}\label{DDE}
c^*&\ge&
\int_{\delta}^{\frac{1}{3}} r^{d-1}e^{2s_{i}m}(|\varphi'|^2+c|\varphi|^2)dr\ge \delta^{d-1}\sum_{n=N(s_{i})}^{+\infty}\left(\int_{x_n}^{y_n}e^{2s_{i}\tilde{m}}|\varphi'|^2dr+c_*\int_{y_n}^{x_{n+1}}e^{2s_{i}\tilde{m}}|\varphi|^2dr\right)\nonumber\\
&=&\delta^{d-1}\sum_{n=N(s_{i})}^{+\infty}\left(e^{-2s_{i}(\frac{1}{6})^n}\int_{x_n}^{y_n}|\varphi'|^2dx
+c_*e^{4s_{i}(\frac{1}{6})^n}\int_{y_n}^{x_{n+1}}|\varphi|^2dx\right)\nonumber\\
&\ge&\delta^{d-1}\sum_{n=N(s_{i})}^{+\infty}\left(e^{-2s_{i}(\frac{1}{6})^n}\frac{|\varphi(s_{i};y_n)-\varphi(s_{i};x_n)|^2}{\alpha^{n+1}}
+\frac{1}{4}\varphi_*^2(\frac{1}{3}) c_*e^{4s_{i}(\frac{1}{6})^n}\beta^{n+1}\right)\nonumber\\
&\ge&\delta^{d-1}\sqrt{c_*}\varphi_*(\frac{1}{3})\inf_{n\ge N(s_i)}\left(\frac{\beta}{\alpha}\right)^{\frac{n+1}{2}}\sum_{n=N(s_{i})}^{+\infty}|\varphi(s_{i};y_n)-\varphi(s_{i};x_n)|.
\end{eqnarray}
In the third inequality, we have utilized the  minimizer of the functional
 \begin{align*}
 \phi''(r)=0,\quad r\in (x_n,y_n); \quad \phi(x_n)=\varphi(s_{i};x_n),\quad\phi(y_n)=\varphi(s_{i};y_n)\nonumber
 \end{align*}
which corresponds to a line segment connecting the points $(x_n,\varphi(s_{i};x_n))$ and $(y_n,\varphi(s_{i};y_n))$. Additional, we have employed the fact that $\varphi(s_i;x)\ge\frac{1}{2}\varphi_*(\frac{1}{3})$ for $n\ge N(s_i)$, and the final inequality follows from the fundamental inequality. \par

Next we will estimate the term $\sum_{n=N(s_{i})}^{+\infty}|\varphi(s_{i};y_n)-\varphi(s_{i};x_n)|$ in \eqref{DDE}.
We note that
there exists $Z_{N(s_{i})-1}\in[y_{N(s_{i})-1},x_{N(s_{i})})$, such that $\varphi(Z_{N(s_{i})-1})\le\varphi_*(\frac{1}{3})\mbox{/}2$ by the definition of $N(s_{i})$. Redefine $\tilde{y}_{N(s_{i})-1}=Z_{N(s_{i})-1}$ and $\tilde{y}_{n}=y_{n}$, for each $n\ge N(s_{i})$. After these modification, we have
\begin{eqnarray}\label{DDD}
&&\sum_{n=N(s_{i})}^{+\infty}|\varphi(s_{i};y_n)-\varphi(s_{i};x_n)|
\ge \sum_{n=N(s_{i})}^{+\infty}\left(|\varphi(s_{i};\tilde{y}_{n})-\varphi(s_{i};\tilde{y}_{n-1})|-|\varphi(s_{i};x_{n})-\varphi(s_{i};\tilde{y}_{n-1})|\right)\nonumber\\
&\ge&\left|\sum_{n=N(s_{i})}^{+\infty}\varphi(s_{i};\tilde{y}_{n})-\varphi(s_{i};\tilde{y}_{n-1})\right|-\sum_{n=N(s_{i})}^{+\infty}\sqrt{\frac{c^*\beta^{n}}{\delta^{d-1}}}
\ge\frac{\varphi_*(\frac{1}{3})}{2}-\sum_{n=N(s_{i})}^{+\infty}\sqrt{\frac{c^*\beta^{n}}{\delta^{d-1}}} ,
\end{eqnarray}
where the second inequality follows from
$$
c^*\ge  \int_{\tilde{y}_{n-1}}^{x_{n}}r^{d-1}e^{2s_{i}\tilde{m}(r)} (|\varphi'|^2+c|\varphi|^2)dr
\ge\delta^{d-1}\int_{\tilde{y}_{n-1}}^{x_{n}} |\varphi'|^2dr
\ge\frac{\delta^{d-1}|\varphi(s_{i};x_{n})-\varphi(s_{i};\tilde{y}_{n-1})|^2}{\beta^{n}}.
$$
Since $\beta\in(0,1)$, according to \eqref{N}, we can choose $i$ large enough such that $\sum_{n=N(s_{i})}^{+\infty}\sqrt{\frac{c^*\beta^{n}}{\delta^{d-1}}}\le \frac{1}{4}\varphi_*(\frac{1}{3})$. In other words, the right-hand side of \eqref{DDD} is strictly positive with positive $\varphi_*(\frac{1}{3})$ and large $i$. By substituting \eqref{DDD} into \eqref{DDE} and considering the fact that $\beta>\alpha$ (as defined by $\tilde{m}(r)$), the right-hand side of \eqref{DDE} tends to infinity as $i\to+\infty$, which results a contradiction. Thus, we conclude that $\varphi_*(\frac{1}{3})=0$, and $\varphi_*(\frac{2}{3})=0$ in light of the symmetric of $m(x)$.
 \end{proof}

We are now at the position to give the lower bound of $\lambda(s)$.
\begin{lemma}[Lower bound]\label{lowerd}
Under conditions in Theorem \ref{dirichlet}, we have
 $$
\liminf_{s\to+\infty}\lambda(s)\ge\lambda^{\mathcal{DD}}.
 $$
 \end{lemma}
 \begin{proof}
Consider the variational characterization of the principal eigenvalue in \eqref{voa} with the transformation $w(s_i;r)=r^{\frac{d-1}{2}}\varphi(s_i;r)$, where the sequence $\{s_i \}_{i=1}^{+\infty}$ is the one obtained in Lemma \ref{converge}. Recall that $m(r)=0$ on $[\frac{1}{3},\frac{2}{3}]$ and then we have the following lower bound estimate of the principal eigenvalue
\begin{eqnarray}
\lambda(s_{i})&=&\int_0^1\left(\left|w'(s_i;r)-sm'(r)w(s_i;r)-\frac{d-1}{2r}w(s_i;r)\right|^2+c(r)|w(s_{i};r)|^2\right)dr\nonumber\\
&\ge&\left(\int_0^{\frac{1}{3}}+\int_{\frac{2}{3}}^1\right)c(r)|w(s_{i};r)|^2dr
+\int_{\frac{1}{3}}^{\frac{2}{3}} r^{d-1}(|\varphi'(s_{i};r)|^2+c(r)|\varphi(s_{i};r)|^2)dr\nonumber\\
&\triangleq& (I_0^{\frac{1}{3}}(s_i)+I_{\frac{2}{3}}^1(s_i))+I_{\frac{1}{3}}^{\frac{2}{3}}(s_i),
\label{estd'}
\end{eqnarray}
where $I_0^{\frac{1}{3}}(s_i)=\int_0^{\frac{1}{3}}c(r)|w(s_{i};r)|^2dr$, $I_{\frac{2}{3}}^1(s_i)=\int_{\frac{2}{3}}^1c(r)|w(s_{i};r)|^2dr$ and $I_{\frac{1}{3}}^{\frac{2}{3}}(s_i)=\int_{\frac{1}{3}}^{\frac{2}{3}} r^{d-1} (|\varphi'(s_{i};r)|^2+c(r)|\varphi(s_{i};r)|^2)dr$.
\par
Next, we estimate the three terms $I_0^{\frac{1}{3}}(s_i)$, $I_{\frac{2}{3}}^1(s_i)$ and $I_{\frac{1}{3}}^{\frac{2}{3}}(s_i)$ respectively.\par
\textbf{Estimate of $I_0^{\frac{1}{3}}(s_i)$:} Let $\epsilon>0$ be a small constant. Define a continuous function $\zeta(r)$ on $[0,1]$ such that $\zeta(r)=c(r)$ for $r\in[0,{\frac{1}{3}}]$; $\zeta(r)=c({\frac{1}{3}})$ for $r\in({\frac{1}{3}}, {\frac{1}{3}}+3\epsilon]$; $\zeta(r)=-\frac{c({\frac{1}{3}})}{\epsilon}(r-{\frac{1}{3}}-4\epsilon)$ for $r\in({\frac{1}{3}}+3\epsilon,{\frac{1}{3}}+4\epsilon]$ and $\zeta(r)=0$ for else. From \eqref{weakc}, we obtain
\begin{eqnarray}
\lim_{i\to+\infty}I_0^{\frac{1}{3}}(s_i)
&\ge&\liminf_{i\to+\infty}\int_{0}^{1}w^2(s_{i};r)\zeta(r)dr
-\limsup_{i\to+\infty}\int_{{\frac{1}{3}}}^{{\frac{1}{3}}+4\epsilon}r^{d-1}\varphi^2(s_{i};r)\zeta(r)dr\nonumber\\
&=&\int_{[0,1]}\zeta(r)\mu(dr)-\limsup_{i\to+\infty}\int_{{\frac{1}{3}}}^{{\frac{1}{3}}+4\epsilon}r^{d-1}\varphi^2(s_{i};r)\zeta(r)dr\nonumber\\
&\ge& \int_{[0,{\frac{1}{3}}+3\epsilon]}\zeta(r)\mu(dr)-C\epsilon
\ge \min_{[0,{\frac{1}{3}}]}c(r)\mu([0,{\frac{1}{3}}+3\epsilon])-C\epsilon,\label{de'}
\end{eqnarray}
where the second inequality follows from the boundedness of $\varphi(s_i;r)$ in \eqref{wb}.\par
\textbf{Estimate of $I_{\frac{2}{3}}^1(s_i)$:} Similarly, we have
\begin{equation}
\lim_{i\to+\infty}I_{\frac{2}{3}}^1(s_i)
\ge \min_{[{\frac{2}{3}},1]}c(r)\mu([{\frac{2}{3}}-3\epsilon,1])-C\epsilon.\label{deb'}
\end{equation}\par

\textbf{Estimate of $I_{\frac{1}{3}}^{\frac{2}{3}}(s_i)$:}
Let us begin by stating an estimate that will be used later. Define a continuous function $\xi(r)$ on $[0,1]$ such $\xi(r)=1$ for $r\in[{\frac{1}{3}}+2\epsilon,{\frac{2}{3}}-2\epsilon]$; $\xi(r)=0$ for $r\in[0,{\frac{1}{3}}+\epsilon]\cup[{\frac{2}{3}}-\epsilon,1]$ and $\xi(r)\in [0,1]$ is continuous for $r\in({\frac{1}{3}}+\epsilon,{\frac{1}{3}}+2\epsilon)\cup({\frac{2}{3}}-2\epsilon,{\frac{2}{3}}-\epsilon)$. Then, according to \eqref{weakc}, we obtain
\begin{eqnarray}\label{same}
&&\liminf_{{i\to+\infty}}\int_{{\frac{1}{3}}}^{{\frac{2}{3}}}r^{d-1}\varphi^2(s_{i};r)dr
=\liminf_{{i\to+\infty}}\int_{{\frac{1}{3}}}^{{\frac{2}{3}}}w^2(s_{i};r)dr
\ge\liminf_{i\to+\infty}\int_{{\frac{1}{3}}+\epsilon}^{{\frac{2}{3}}-\epsilon}w^2(s_{i};r)\xi(r)dr\nonumber\\
&=&\int_{[{\frac{1}{3}}+\epsilon,{\frac{2}{3}}-\epsilon]}\xi(r)\mu(dr)\ge\mu([{\frac{1}{3}}+2\epsilon,{\frac{2}{3}}-2\epsilon]).
\end{eqnarray}

Next we will use \eqref{same} to estimate $I_{\frac{1}{3}}^{\frac{2}{3}}(s_i)$. Without loss of generality, assume that  for large $i$,
\begin{equation}
\int_{\frac{1}{3}}^{\frac{2}{3}}r^{d-1}|\varphi(s_{i};r)|^2dr>0.\label{assdd}
\end{equation}
Then, we can denote that
$$
v(s_{i};r)=\frac{\varphi(s_{i};r)}{\sqrt{\int_{\frac{1}{3}}^{\frac{2}{3}}r^{d-1}|\varphi(s_{i};r)|^2dr}}.
$$
It is easy to see that
$$
v(s_{i};r)\to \frac{\varphi_*(r)}{\sqrt{\int_{\frac{1}{3}}^{\frac{2}{3}}r^{d-1}|\varphi_*(r)|^2dr}}\triangleq v_1(r),\ \text{in }C_{loc}^1\left(\left[{\frac{1}{3}},{\frac{2}{3}}\right]\right)\text{  as }i\to+\infty,
$$
where $\varphi_*(r)$ is determined in Lemma \ref{id} and $\int_{\frac{1}{3}}^{\frac{2}{3}}r^{d-1}v_1^2(r)dr=1$. Since $\varphi_*({\frac{1}{3}})=\varphi_*({\frac{2}{3}})=0$, we know
 $v_1({\frac{1}{3}})=v_1({\frac{2}{3}})=0$ .\par
Now we can derive
\begin{eqnarray}
&&\lim_{i\to+\infty}I_{\frac{1}{3}}^{\frac{2}{3}}(s_i)
=\lim_{i\to+\infty}\int_{\frac{1}{3}}^{\frac{2}{3}}r^{d-1}|\varphi(s_{i};r)|^2dr\int_{\frac{1}{3}}^{\frac{2}{3}}r^{d-1}\left(|v'(s_{i};r)|^2+c(r)|v(s_{i};r)|^2\right)dr\nonumber\\
&\ge&\mu\left(\left[{\frac{1}{3}}+2\epsilon,{\frac{2}{3}}-2\epsilon\right]\right)\int_{\frac{1}{3}}^{\frac{2}{3}}r^{d-1}\left(|v_1'(r)|^2+c(r)|v_1(r)|^2\right)dr\nonumber\\
&\ge&\mu\left(\left[{\frac{1}{3}}+2\epsilon,{\frac{2}{3}}-2\epsilon\right]\right)\lambda^{\mathcal{DD}},\label{ddab}
\end{eqnarray}
by \eqref{same} and the variational characterization of $\lambda^{\mathcal{DD}}$.

In fact, if \eqref{assdd} does not holds true, then there exists a subsequence $\{s_{i}\}_{i=1}^{+\infty}$, labelled by itself for simplicity, satisfying
$$
\lim_{i\to+\infty}\int_{\frac{1}{3}}^{\frac{2}{3}}r^{d-1}\varphi^2(s_{i};r)dr=\lim_{i\to+\infty}\int_{\frac{1}{3}}^{\frac{2}{3}}w^2(s_{i};r)dr=0.
$$
Due to \eqref{weakc} and the fact
$$
\mu\left(\left({\frac{1}{3}},{\frac{2}{3}}\right)\right)=\sup\left\{\mu\left(\left[\tilde{{\frac{1}{3}}},\tilde{{\frac{2}{3}}}\right]\right):\ \left[\tilde{{\frac{1}{3}}},\tilde{{\frac{2}{3}}}\right]\subset \left({\frac{1}{3}},{\frac{2}{3}}\right)\right\},
$$
 we have $\mu(({\frac{1}{3}},{\frac{2}{3}}))=0$, which leads to
\begin{equation}\label{mud'}
\mu\left(\left[0,{\frac{1}{3}}\right]\right)+\mu\left(\left[{\frac{2}{3}},1\right]\right)=1.
\end{equation}
By using Lemma \ref{upd}, \eqref{estd'}, \eqref{de'}, \eqref{deb'}, \eqref{mud'} under the assumption (H1), we can deduce that
\begin{eqnarray*}
\lambda^{\mathcal{DD}}\ge\lambda^*\ge \min_{[0,{\frac{1}{3}}]\cup[{\frac{2}{3}},1]}c(r)\left(\mu\left(\left[0,{\frac{1}{3}}+3\epsilon\right]\right)+\mu\left(\left[{\frac{2}{3}}-3\epsilon,1\right]\right)\right)-C\epsilon>\lambda^{\mathcal{DD}},
\end{eqnarray*}
by selecting $\epsilon> 0$ to be sufficiently small. This leads to a contradiction.

\textbf{Estimate of $\lambda(s)$:}
Combining the estimated results \eqref{de'}, \eqref{deb'} and \eqref{ddab} in \eqref{estd'}, we get
\begin{eqnarray}
\lim_{i\to+\infty}\lambda(s)&\ge&\lim_{i\to+\infty}(I_0^{\frac{1}{3}}(s_i)+I_{\frac{1}{3}}^{\frac{2}{3}}(s_i)+I_{\frac{2}{3}}^1(s_i))\nonumber\\
&\ge&\min_{\left[0,{\frac{1}{3}}\right]\cup[{\frac{2}{3}},1]}c(r)\left(\mu\left(\left[0,{\frac{1}{3}}+3\epsilon\right]\right)+\mu\left(\left[{\frac{2}{3}}-3\epsilon,1\right]\right)\right)
+\mu\left(\left[{\frac{1}{3}}+2\epsilon,{\frac{2}{3}}-2\epsilon\right]\right)\lambda^{\mathcal{DD}}-C\epsilon\nonumber\\
&\ge&\lambda^{\mathcal{DD}}-C\epsilon\label{ld'}
\end{eqnarray}
where the last inequality holds by  (H1). Taking the limit as $\epsilon \to 0$ on the both sides of \eqref{ld'}, we conclude the proof.\par
\end{proof}
\textbf{Proof of Theorem \ref{dirichlet}:} Combining Lemma \ref{upd}, Lemma \ref{id} and Lemma \ref{lowerd}, we complete our proof.

\section{Limit of the principal eigenvalue for $m(x)\in S_{\mathcal{NN}}$}
 \begin{theorem}\label{neumann}
 Suppose that (H2) holds and  $m(x)\in S_{\mathcal{NN}}$, then   the principal eigenvalue $\lambda(s)$ of \eqref{eq}  satisfies
 $$\lim_{s\to+\infty}\lambda(s)=\lambda^{\mathcal{NN}}$$
\end{theorem}


We first consider the lower bound of $\lambda(s)$.
\begin{lemma}[Lower bound]\label{nnl}
Under the conditions in Theorem \ref{neumann}, we have
$$\liminf_{s\to+\infty}\lambda(s)\ge \lambda^{\mathcal{NN}}.$$
\end{lemma}

\begin{proof}
Recall that we denote $\varphi(s;r)$ as the principal eigenfunction of \eqref{eq} with respect to $\lambda(s)$ for any fixed $s>0$. 
By the variational characterization \eqref{voe} and the transformation $w(s_{i};r)=r^{\frac{d-1}{2}}e^{sm(r)}\varphi(s_{i};r)$, 
we obtain
\begin{eqnarray}\label{ne}
\lambda(s_{i})&=&\int_0^1r^{d-1}e^{2sm(r)}(|\varphi'(s_{i};r)|^2+c(r)|\varphi(s_{i};r)|^2)dr\nonumber\\
&\ge&\left(\int_{0}^{\frac{1}{3}}+\int_{\frac{2}{3}}^{1}\right)c(r)w^2(s_{i};r)dx+\int_{\frac{1}{3}}^{\frac{2}{3}}r^{d-1} (|\varphi'(s_{i};r)|^2+c(r)|\varphi(s_{i};r)|^2)dr.
\end{eqnarray}
For the first two terms in \eqref{ne}, we can directly use \eqref{de'} and \eqref{deb'} to get
\begin{equation}\label{nne}
\liminf_{i\to+\infty}\left(\int_{0}^{\frac{1}{3}}+\int_{\frac{2}{3}}^{1}\right)c(r)w^2(s_{i};r)dr\ge\min_{[0,{\frac{1}{3}}]\cup[{\frac{2}{3}},1]}c(r)\left[\mu([0,{\frac{1}{3}}+3\epsilon])+\mu([{\frac{2}{3}}-3\epsilon,1])\right]-C\epsilon.
\end{equation}

For the last term in \eqref{ne}, it is sufficient to prove that
\begin{equation}\label{abnn}
\liminf_{i\to+\infty}\int_{\frac{1}{3}}^{\frac{2}{3}} r^{d-1} (|\varphi'(s_{i};r)|^2+c(r)|\varphi(s_{i};r)|^2)dr
\ge\lambda^{\mathcal{NN}}\mu([{\frac{1}{3}}+2\epsilon,{\frac{2}{3}}-2\epsilon]).
\end{equation}
By taking the limit inferior on both sides of \eqref{ne}, using the estimates in \eqref{nne}, \eqref{abnn}, and letting $\epsilon \to 0$, we obtain the desired result, $\liminf_{s\to+\infty}\lambda(s)\ge \lambda^{\mathcal{NN}}$ by the assumption in (H2).
\par

Next, we aim to prove the inequality \eqref{abnn}. With \eqref{assdd}, we could define
$$
v(s_{i};r)=\frac{\varphi(s_{i};r)}{\sqrt{\int_{\frac{1}{3}}^{\frac{2}{3}}r^{d-1}\varphi^2(s_{i};r)dr}}\in H^1([{\frac{1}{3}},{\frac{2}{3}}]).
$$
And it is easy to check that $\int_{\frac{1}{3}}^{\frac{2}{3}}r^{d-1}v^2(s_i;r)dr=1$.
Then, we can write that
\begin{eqnarray*}
\int_{\frac{1}{3}}^{\frac{2}{3}}r^{d-1}(|\varphi'(s_{i};r)|^2+c|\varphi(s_{i};r)|^2)dr
&=&\int_{\frac{1}{3}}^{\frac{2}{3}}r^{d-1}\varphi^2(s_{i};r)dr\int_{\frac{1}{3}}^{\frac{2}{3}}r^{d-1} (|v'(s_{i};r)|^2+c|v(s_{i};r)|^2)dr.
\end{eqnarray*}
By taking the limit inferior on both sides of the aforementioned equality, utilizing the estimate in  \eqref{same} and considering the characterization of $\lambda^{\mathcal{NN}}$, we can establish \eqref{abnn} and thereby complete our proof.
\end{proof}

Before presenting the proof of the upper bound, we first introduce the definition of the test function. For any function $m(r)\in S_{\mathcal{NN}}$, there exist constants $\delta$ and $\alpha$ such that $m(r)\leq \bar{m}(\alpha;r)$ on $[\delta,{\frac{1}{3}})\cup({\frac{2}{3}},1-\delta]$.
With these constants, we have
\begin{definition}\label{dnn}
Definition of $\sigma_n(s)$ and $l_n(s)$.
\begin{equation}
\sigma_n(s)\triangleq\alpha^n e^{3  ({\frac{1}{6}})^{n}s }
\text{ \quad
and \quad}
 l_n(s)\triangleq
\frac { 1 } { \prod _ { k = n } ^ {+ \infty } \left( 1 + \alpha^k e^{ 3({\frac{1}{6}})^{k}s} \right) }=\frac { 1 } { \prod _ { k = n } ^ {+ \infty } \left( 1 + \sigma_k(s) \right) }.
\end{equation}
\end{definition}
\begin{remark}\label{propnn}
From Definition \ref{dnn}, it is easy to check that\par
(1) $l_n(s)$ is increasing in $n$, $\lim_{n\to +\infty}l_n(s)=1$ for any fixed $s>0$ and $\lim_{s\to +\infty}l_n(s)=0$ for any fixed $n\in \mathbb{N}$;\par
(2) $l_{n+1}(s)-l_n(s)=\sigma_n(s)l_n(s)$.
\end{remark}

With these preparations, we can now proceed to provide an upper bound estimate.

\begin{lemma}[Upper bound]\label{nnu}
Under the conditions in Theorem \ref{neumann}, we have
$$\limsup_{s\to+\infty}\lambda(s)\le \lambda^{\mathcal{NN}}.$$
\end{lemma}

\begin{proof}
Note that $m(r)\in S_{\mathcal{NN}}$ and there exist the constants $\delta$, $\alpha$ such that $m(r)\leq\bar{m}(\alpha;r)$ on $[\delta,{\frac{1}{3}})\cup({\frac{2}{3}},1-\delta]$. Recall that $\bar{m}(\delta)={\frac{1}{6}}$, so by the continuity of $m(r)$, there exists a small positive constant $\delta_1$ such that $m(r)\leq \frac{1}{4}$ for $r\in (\delta-\delta_1,\delta]$.
We define the test function $\varphi_1\in H^1([0,1])$ as
\begin{equation}\label{var1}
\varphi _ { 1 }(s;r) = \left\{ \begin{array}{lcl}
0, && r\in[0,\delta-\delta_1),\\
\frac{\varphi^{\mathcal{NN}}({\frac{1}{3}})l_1(s)}{\delta_1}(r-\delta+\delta_1), && r\in[\delta-\delta_1,\delta),\\
\varphi^{\mathcal{NN}}({\frac{1}{3}})l_n(s), & & r\in [Y_{n-1}, X_{n}),\\
\varphi^{\mathcal{NN}}({\frac{1}{3}})\left(l_n(s)+\frac { l_{n+1}(s) - l _ { n }(s) } { \alpha ^{ n } }(r-X_n)\right), & & r\in [X_{n}, Y_n), \,\, \\
\varphi^{\mathcal{NN}}(r),&& r\in[{\frac{1}{3}},{\frac{2}{3}}],\\
\frac{\varphi^{\mathcal{NN}}({\frac{2}{3}})}{\varphi^{\mathcal{NN}}({\frac{1}{3}})} \varphi_1(s;1-r),& & r\in({\frac{2}{3}},1], \end{array} \right.
\end{equation}
where
$\varphi^{\mathcal{NN}}(r)$ is the bounded positive principal eigenfunction corresponding to $\lambda^{\mathcal{NN}}$ in Definition \ref{d2}, the points $\{X_n\}_{n=1}^{+\infty}$, $\{Y_{n-1}\}_{n=1}^{+\infty}$ are defined in Definition \ref{DM} and $l_n(s)$ is in Definition \ref{dnn}. Then, by the variational characterization of $\lambda(s)$ in \eqref{voa}, we have
\begin{eqnarray}
\lambda(s)
&\le&\frac{ \left(\int_0^{\frac{1}{3}}+\int_{\frac{1}{3}}^{\frac{2}{3}}+\int_{\frac{2}{3}}^1\right)r^{d-1}e^{2sm(r)}\left(|\varphi_1'(s;r)|^2+c|\varphi_1(s;r)|^2\right)dr} {\int_{\frac{1}{3}}^{\frac{2}{3}}r^{d-1}|\varphi_1(s;r)|^2dr}\nonumber\\
&=& \lambda^{\mathcal{NN}}+ \frac{ I_0^{\frac{1}{3}}[\varphi_1](s)+I_{\frac{2}{3}}^1[\varphi_1](s)}{\int_{\frac{1}{3}}^{\frac{2}{3}}r^{d-1}|\varphi_1|^2dr},\label{lbd1}
\end{eqnarray}
where we use the notation $$I_{e_1}^{e_2}[\varphi](s)\triangleq\int_{e_1}^{e_2}r^{d-1}e^{2sm(r)}\left(|\varphi'(r)|^2+c(r)|\varphi(r)|^2\right)dr,$$
for any function $\varphi(r)\in H^1([0,1])$ and two points $e_1<e_2$ on $(0,1)$.
\par
We consider $I_0^{\frac{1}{3}}[\varphi_1](s)$ first (see that $I_{\frac{2}{3}}^1[\varphi_1](s)$ is just the symmetric of $I_0^{\frac{1}{3}}[\varphi_1](s)$) and we have
\begin{eqnarray}
I_0^{\frac{1}{3}}[\varphi_1](s)
&=& \int_{0}^{\delta}r^{d-1}e^{2sm(r)}\left(|\varphi_1'|^2+c(r)|\varphi_1|^2\right)dr+\int_\delta^{\frac{1}{3}} r^{d-1}e^{2sm(r)}\left(|\varphi_1'|^2+c(r)|\varphi_1|^2\right)dr\nonumber\\
&\le&\int_{\delta-\delta_1}^{\delta} r^{d-1} e^{2sm(r)}\left(|\varphi_1'|^2+c(r)|\varphi_1|^2\right)dr+\sum_{n=1}^{+\infty} \int_{Y_{n-1}}^{Y_{n}}r^{d-1}e^{2s\bar{m}(x)}\left(|\varphi_1'|^2+c(x)|\varphi_1|^2\right)dx\nonumber\\
&\le&C\left[e^{\frac{1}{2}s}l_1^2(s)+ \sum_{n=1}^{+\infty} \alpha^{n} e^{2(\frac{1}{6})^{n}s}l_n^2
+\sum_{n=1}^{+\infty}\left(e^{-4(\frac{1}{6})^{n}s} \frac{(l_{n+1}-l_n)^2}{\alpha^{n}}+\alpha^{n}e^{-4 (\frac{1}{6})^{n}s}l_{n+1}^2\right)\right]\nonumber\\
&\le&C\left(e^{\frac{1}{2}s}l_1^2(s)+\sum_{n=1}^{+\infty}\alpha^{n}e^{2(\frac{1}{6})^{n}s}l_n^2+\sum_{n=1}^{+\infty}\alpha^{n}e^{-4 (\frac{1}{6})^{n}s}l_{n+1}^2\right)\nonumber\\
&\triangleq& C(E(s)+F(s)+G(s)),\label{est E(s) 0}
\end{eqnarray}
where $C=C(\varphi^{\mathcal{NN}}(\frac{1}{3}), \varphi^{\mathcal{NN}}(\frac{2}{3}), d,  c^*)$ is a positive constant and the last inequality follows from (2) of Remark \ref{propnn}, and we denote the following notations
\begin{eqnarray}\label{FG}
E(s)=e^{\frac{1}{2}s}l_1^2(s),\quad
F(s)
=\sum_{n=1}^{+\infty} \alpha^{n}e^{2(\frac{1}{6})^{n}s}l_n^2,
\quad
G(s)=\alpha^{n}e^{-4 (\frac{1}{6})^{n}s}l_{n+1}^2.
\end{eqnarray}\par
To study the upper bound of \eqref{est E(s) 0}, we estimate three functions in \eqref{FG} respectively.\par
\textbf{Estimate of $E(s)$:}
From the definition of $l_1(s)$ (Remark \ref{propnn} (2)), we get
\begin{equation}
E(s)=\frac { e^{\frac{1}{2}s} } { \prod _ { k = 1 } ^ {+ \infty } \left( 1 + \alpha^ke^{3(\frac{1}{6})^{k}s} \right)^2 }\to 0, \quad \text{ as } s\to+\infty.
\end{equation}

\textbf{Estimates for $F(s)$:}
By definition of $\sigma_n(s)$ in Definition \ref{dnn}, we could say that for any constant  $\epsilon\in (0, 1)$, there exist integers $0<K_1(s,\epsilon)\le K_2(s,\epsilon)$ such that
\begin{equation}\label{K1K2}
\left\{ \begin{array} { l } \sigma_1(s),\ \sigma_2(s),...\ \sigma_{K_1}(s)\in (\frac{1}{\epsilon},+\infty), \\
\sigma_{K_1+1}(s),\ \sigma_{K_1+2}(s),...\ \sigma_{K_2}(s)\in (\epsilon, \frac{1}{\epsilon})\\
\sigma_{K_2+1}(s),\ \sigma_{K_2+2}(s),...\in (0,\epsilon), \end{array} \right.
\end{equation}
where we use the notation $K_1$ and $K_2$ for simplicity but keep in mind that for any fixed $\epsilon$,  $K_1$ and $K_2$ are functions depending on $s$. Based on this, we insert $l_n(s)$ into   $F(s)$ in \eqref{FG} and divide it into three parts
\begin{eqnarray}\label{DF}
F(s)
=\left(\sum_{n=1}^{K_1}+\sum_{n=K_1+1}^{K_2}+\sum_{n=K_2+1}^{+\infty}\right)\frac { \alpha^{n}e^{2(\frac{1}{6})^{n}s}} { \prod _ { k = n } ^ { +\infty } \left( 1 + \sigma _ { k }(s) \right)^2 }
\triangleq F_1(s)+F_2(s)+F_3(s).
\end{eqnarray}\par

We give the upper bound of $F_1(s)$, $F_2(s)$ and $F_3(s)$ in the following, respectively.\par

For $F_1(s)$,  noting that $\sigma_n(s)=\alpha^ne^{3(\frac{1}{6})^ns}$, we have
\begin{eqnarray}
F_1(s)
&=&\sum_{n=1}^{K_1}\frac{\alpha^{n}e^{2(\frac{1}{6})^{n}s}}{ \prod _ { k = n } ^ { +\infty } \left( 1 + \sigma _ { k }(s) \right)^2}
=\sum_{n=1}^{K_1}\frac{\sigma_ne^{-s(\frac{1}{6})^ns}}{ \prod _ { k = n } ^ { +\infty } \left( 1 + \sigma _ { k }(s) \right)^2}
\le\sum_{n=1}^{K_1}\frac{\sigma_n}{ \prod _ { k = n } ^ { +\infty } \left( 1 + \sigma _ { k }(s) \right)^2}\nonumber\\
&\le&\frac{1}{1+\sigma_{K_1}}\sum_{n=1}^{K_1}\frac{( \sigma_n+1)-1}{ \prod _ { k = n } ^ { +\infty } \left( 1 + \sigma _ { k }(s) \right)}
\le\frac{1}{1+\frac{1}{\epsilon}}\sum_{n=1}^{K_1}(l_{n+1}(s)-l_n(s))
\le\frac{1}{1+\frac{1}{\epsilon}}
\le \epsilon,\label{F1}
\end{eqnarray}
where  the third inequality follows from the fact $\sigma_n(s)>\frac{1}{\epsilon}$ for $n\le K_1$ in \eqref{K1K2} and the last inequality follows from
(1) of Remark \ref{propnn}.
 \par

Then we consider $F_2(s)$.
First we claim that for sufficiently large $s$,
there is at most one term $\sigma_n$ lying in the interval $(\epsilon, \frac{1}{\epsilon})$. In fact, if not, denote $K=K_1+1$ and we have
\begin{equation}
\epsilon<\alpha^{K+1}e^{3(\frac{1}{6})^{K+1}s}<\alpha^Ke^{3(\frac{1}{6})^Ks}<\frac{1}{\epsilon} \label{one}
\end{equation}
for large $s$.
From the last inequality of \eqref{one}, we get
\begin{equation}
s<\frac{1}{3(\frac{1}{6})^{K} }\ln\frac{1}{\epsilon\alpha^K},\label{s1}
\end{equation}
which means that $K\to+\infty$ as $s\to+\infty$. Substituting \eqref{s1} into the first inequality of \eqref{one}, we obtain
\begin{eqnarray}
\epsilon&<&\alpha^{K+1}e^{\frac{1}{6}\ln\frac{1}{\epsilon\alpha^K}}=\frac{\alpha^{\frac{5}{6}K+1}}{\epsilon^{\frac{1}{6}}}.\label{con1}
\end{eqnarray}
Since $\alpha\in(0,1)$, the right hand side of \eqref{con1} tends to zero as $s\to+\infty$, which leads to a contradiction. So the only possible term $\sigma_K\in(\epsilon,\frac{1}{\epsilon})$ may determine the bound of $s$, namely,
\begin{equation}
\frac{1}{3(\frac{1}{6})^{K} }\ln\frac{\epsilon}{\alpha^K}<s<\frac{1}{3(\frac{1}{6})^{K} }\ln\frac{1}{\epsilon\alpha^K}.\label{s}
\end{equation}
Thus for large $s$, we obtain
\begin{eqnarray}
F_2(s)&=&\frac { \alpha^Ke^{2(\frac{1}{6})^Ks} } { \prod _ { k = K } ^ { +\infty } \left( 1 + \sigma _ { k }(s) \right)^2 }
=\frac{\sigma_Ke^{-(\frac{1}{6})^Ks}}{ \prod _ { k = K } ^ { +\infty } \left( 1 + \sigma _ { k }(s) \right)^2}\le\frac{1}{\epsilon}e^{-\frac{1 }{3 }\ln \frac{\epsilon}{\alpha^K}}=\frac{1}{\epsilon}\left(\frac{\alpha^K}{\epsilon}\right)^{\frac{1}{3}},\label{F2}
\end{eqnarray}
where the inequality follows from $\sigma_K\in(\epsilon,\frac{1}{\epsilon})$ and the first inequality of \eqref{s}. If there is no term $\sigma_K$ in $(\epsilon,\frac{1}{\epsilon})$, $F_2(s)=0$ still satisfies \eqref{F2}.
\par

For $F_3(s)$, when $n\ge K_2+1$, we have $\sigma_n\in(0,\epsilon)$ from \eqref{K1K2}, and then
\begin{equation}\label{k2}
\sigma_n=\alpha^ne^{3(\frac{1}{6})^{n}s}\le \alpha^{n-K_2-1}\sigma_{K_2+1}\le\alpha^{n-K_2-1}\epsilon.
\end{equation}
Thus,  we obtain
\begin{eqnarray}
F_3(s)
=\sum_{n=K_2+1}^{+\infty}\frac{\alpha^{n}e^{2(\frac{1}{6})^{n}s}}{ \prod _ { k = n } ^ { +\infty } \left( 1 + \sigma _ { k }(s) \right)^2}
\le\sum_{n=K_2+1}^{+\infty}\sigma_n
\le\epsilon\sum_{n=K_2+1}^{+\infty}\alpha^{n-K_2-1}
\le C\epsilon,\label{F3}
\end{eqnarray}
where the last inequality follows from the assumption $\sum_{n=1}^{+\infty}\alpha^n<1$ in Definition \ref{e1}.\par

Hence,
combining \eqref{F1}, \eqref{F2} with \eqref{F3}, we obtain the estimate of $F(s)$ for large $s$:
\begin{eqnarray}\label{F}
F(s)=F_1(s)+F_2(s)+F_3(s)
\le\epsilon+\frac{1}{\epsilon}\left(\frac{\alpha^K}{\epsilon}\right)^{\frac{1}{3}}+C\epsilon.
\end{eqnarray}
Let $s\to +\infty$ first and $\epsilon\to 0$  next in \eqref{F}, we acquire $F(s)\to 0$ with $K\to+\infty$ as $s\to+\infty$.\par

\textbf{Estimate of $G(s)$:} Applying  Lebesgue's dominated convergence theorem with Remark \ref{propnn}, we have
\begin{equation}\label{les}
\lim_{s\to+\infty}G(s)=\lim_{s\to+\infty}\sum_{n=1}^{+\infty}\alpha^{n}e^{-4 (\frac{1}{6})^{n}s}l_{n+1}^2=\sum_{n=1}^{+\infty}\lim_{s\to+\infty}\alpha^{n}e^{-4 (\frac{1}{6})^{n}s}l_{n+1}^2=0.
\end{equation}
\par
Therefore, by the estimates of $E(s)$, $F(s)$ and $G(s)$, we conclude $I_0^{\frac{1}{3}}[\varphi_1](s)\to0$ as $s\to+\infty$. Similarly,  $I_{\frac{2}{3}}^1[\varphi_1](s)\to0$ as $s\to+\infty$. Take the upper limit at both of the sides of \eqref{lbd1} and we complete the proof.
\end{proof}

\textbf{Proof of Theorem \ref{neumann}:} Combining Lemma \ref{nnl} and Lemma \ref{nnu}, we finish the proof.

\section{Proof of Theorem \ref{robin}}

In this section, we provide a counterexample that demonstrates the non-existence of the limit of principal
eigenvalue $\lambda(s,m(r))$ as $s\to+\infty$, thereby proving Theorem \ref{robin}.
To establish this result, we begin by introducing a lemma that illustrates the continuity dependence of the principal eigenvalue $\lambda(s,m(r))$ on $m(r)$. Subsequently, we utilize Theorem \ref{dirichlet}, Theorem \ref{neumann} and the aforementioned lemma to fully establish the proof of Theorem \ref{robin}.
\begin{lemma}\label{lemma}
For the principal eigenvalue $\lambda(s,m(r))$ of \eqref{eq}, with $m(r)=m_i(r)$, $i=1,2$,
it holds that
\begin{equation}
\left|\lambda(s,m_1(r))- \lambda(s,m_2(r)) \right|\le c^*\left(e^{4s||m_1-{m_2}||_{C([0,1])}}-1\right),\label{rbes}
\end{equation}
where $c^*=\max_{[0,1]}c(r)$.
\end{lemma}
\begin{proof}
 By the variational characterization \eqref{voa}, we have
\begin{eqnarray}
&&\left|\lambda(s,m_1(r))- \lambda(s,m_2(r)) \right|\nonumber\\
&=&\min_{\int_0^1 r^{d-1} e^{2sm_1(r)}|\phi(r)|^2dr=1}\int_0^1 r^{d-1} e^{2sm_1(r)}\left[|\phi'(r)|^2+(c(r)-\lambda(s,m_2(r)))|\phi(r)|^2\right]dr,\label{vart}
\end{eqnarray}
for $\phi\in H^1([0,1])$.
Taking the test function $\phi(r)=\varphi_2(r)$, that is the principal eigenfunction corresponding to \eqref{eq} with  $m(r)=m_2(r)$, from \eqref{vart}, we obtain
\begin{eqnarray}
&&\left|\lambda(s,m_1(r))- \lambda(s,m_2(r)) \right|\nonumber\\
&\le&\frac{\int_0^1 r^{d-1} e^{2sm_1(r)}\left(|\varphi_2'(r)|^2+(c(r)-\lambda(s,m_2(r)))|\varphi_2(r)|^2\right)dr}{\int_0^1r^{d-1}e^{2sm_1(r)}|\varphi_2(r)|^2dr}\nonumber\\
&\le& e^{4s||m_1-{m_2}||_{C([0,1])}}\frac{\int_0^1 r^{d-1} e^{2sm_2(r)}\left(|\varphi_2'(r)|^2+c(r)|\varphi_2(r)|^2\right)dr}{\int_0^1 r^{d-1} e^{2sm_2(r)}|\varphi_2(r)|^2dr}-\lambda(s,m_2(r))\nonumber\\
&\le& \lambda(s,m_2(r))\left(e^{4s||m_1-{m_2}||_{C([0,1])}}-1\right)
\le c^*\left(e^{4s||m_1-{m_2}||_{C([0,1])}}-1\right),\label{p}
\end{eqnarray}
where the last inequality follows from \eqref{ubd}. By changing the subscript number, we conclude \eqref{rbes}.
\end{proof}
Recalling Definition \ref{DM} and \ref{e1}, we have
\begin{remark}\label{fold}
If $m(r)\in S_{\mathcal{DD}}$, then there exist $\delta> 0$ and a $C^1([0,1])$ function
\begin{equation*}
m^*(r)=\left\{
\begin{array}{l}
m(r),\quad r\in[0,\delta)\cup (1-\delta,1]\\
-m(r),\quad r\in[\delta,\frac{1}{3})\cup (\frac{2}{3},1-\delta]\\
0,\quad r\in[\frac{1}{3},\frac{2}{3}]
\end{array}\right.
\end{equation*}
satisfying $m^*(r)\in S_{\mathcal{NN}}$. 
\end{remark}
We are now at the position to give the proof of our main result.\par
\textbf{Proof of Theorem \ref{robin}:}

It is indeed assumed that the counterexample $\hat{m}(r)$ satisfies $\hat{m}(r)=\hat{m}(1-r)$ for $r\in[0,1]$ in Section \ref{S:pre}. In the subsequent construction, we will fucus on constructing $\hat{m}(r)$ on the interval $[0,\frac{1}{3}]$ and the part on $[\frac{2}{3},1]$ can be obtained through symmetry. To facilitate the understanding of the construction process, we will divide it into three steps.

\textbf{Step 1:} (Construction of a sequence of functions $\{m_n(r)\}_{n=1}^{+\infty}$.)

Firstly, we construct a sequence of functions $\{m_n(r)\}_{n=1}^{+\infty}$ such that the limit as $n\to+\infty$ is the desired function $\hat{m}(r)$. 

We select a function $m_1(r)\in S_{\mathcal{DD}}$ such that
$m_1(r)\ge 0$ on $[x_n,y_n]$ with $m_1(z_n)=m_1'(z_n)=0$ where $z_n$ is the mid-point of $x_n$ and $y_n$ for $n\in \mathbb{N}$ and the points $x_n$, $y_n$ are ones defined in $S_{\mathcal{DD}}$.
It is easy to see that this type $m_1(r)$ exists 
 and  from Theorem \ref{dirichlet}, we have
$$\lim_{s\to+\infty}\lambda(s,m_1(r))=\lambda^{\mathcal{DD}}.$$
 Thus, 
 there exists $s_1>8\mbox{/}\ln(1+\rho\mbox{/}2c^*)$ such that
 $$|\lambda(s_1,m_1(r))-\lambda^{\mathcal{DD}}|<\frac{\lambda^{\mathcal{DD}}-\lambda^{\mathcal{NN}}}{2}\triangleq\frac{\rho}{2},$$
 where $\rho>0$ is a constant. 
 For the given $s_1$, let $\tau_1=1\mbox{/}s_1^2$, by the continuity of $m_1(r)$ and the fact $m_1(\frac{1}{3})=0$, there exists $\delta(\tau_1)>0$, such that $|m_1(r)|<\tau_1$, as $r\in(\frac{1}{3}-\delta(\tau_1),\frac{1}{3})$. \par
  We denote $\kappa(\tau_1)=\inf\{z_n|z_n>\frac{1}{3}-\delta(\tau_1)\}$ and define
$$
m_2(r)=\left\{
\begin{array}{l}
m_1(r),\quad r\in(0,\kappa(\tau_1)],\\
-m_1(r),\quad r\in(\kappa(\tau_1),\frac{1}{3}).
\end{array}\right.
$$
Note that $m_2(r)\in C^1([0,1])$ with
 $|m_1(r)-m_2(r)|\le 2\tau_1$. From Lemma \ref{lemma} we obtain
$$
|\lambda(s_1,m_2(r))-\lambda(s_1,m_1(r))|\le c^*(e^{\frac{8}{s_1}}-1)<\frac{\rho}{2},
$$
and
$$
|\lambda(s_1,m_2(r))-\lambda^{\mathcal{DD}}|\le |\lambda(s_1,m_2(r))-\lambda(s_1,m_1(r))|+|\lambda(s_1,m_1(r))-\lambda^{\mathcal{DD}}|\le \rho,
$$
by the triangle inequality.
Moreover, $m_2(r)\in S_{\mathcal{NN}}$ from Remark \ref{fold} and
then  we obtain from Theorem \ref{neumann} that
$$\lim_{s\to+\infty}\lambda(s,m_2(r))=\lambda^{\mathcal{NN}}.$$
 Thus, there exists $s_2>\max\{s_1,8\mbox{/}\ln(1+\rho\mbox{/}3c^*)\}$ such that $$|\lambda(s_2,m_2(r))-\lambda^{\mathcal{NN}}|<\frac{\rho}{3}.$$
 For the given $s_2$, let $\tau_2=1\mbox{/}s_2^2$, by the continuity of $m_2(r)$, there exists $\delta(\tau_2)\in(0,\delta(\tau_1))$, such that $|m_2(r)|<\tau_2$, as $r\in(\frac{1}{3}-\delta(\tau_2),\frac{1}{3})$.
 \par

 Define $\kappa(\tau_2)=\inf\{z_n|z_n>\frac{1}{3}-\delta(\tau_2)\}$, and
$$
m_3(r)=\left\{
\begin{array}{l}
m_2(r),\quad r\in(0,\kappa(\tau_2)],\\
-m_2(r),\quad r\in(\kappa(\tau_2),\frac{1}{3}).
\end{array}\right.
$$
Thus $m_3(r)\in C^1([0,1])$ with $|m_2(r)-m_3(r)|\le 2\tau_2$. From Lemma \ref{lemma}, we obtain
$$
|\lambda(s_2,m_3(r))-\lambda(s_2,m_2(r))|\le c^*(e^{\frac{8}{s_2}}-1)<\frac{\rho}{3}.
$$
And by the triangle inequality
$$
|\lambda(s_2,m_3(r))-\lambda^{\mathcal{NN}}|\le |\lambda(s_2,m_3(r))-\lambda(s_2,m_2(r))|+|\lambda(s_2,m_2(r))-\lambda^{\mathcal{NN}}|\le \frac{2\rho}{3}.
$$
Moreover, $m_3(r)\in S_{\mathcal{DD}}$ by the constructions of $m_1(r)$ and $m_2(r)$, and
then  we obtain from Theorem \ref{dirichlet} that
$$\lim_{s\to+\infty}\lambda(s,m_3(r))=\lambda^{\mathcal{DD}}.$$
\par
Repeat the procedure above, we derive:

For $\{m_{2n+1}(r)\}_{n=1}^{+\infty}$ obtained in this construction, we have $m_{2n+1}(r)\in S_{\mathcal{DD}}$ and
there exists $s_{2n}>\max\{s_{2n-1},8\mbox{/}\ln(1+\rho\mbox{/}(2n+1)c^*)\}$ such that
\begin{equation}
|\lambda(s_{2n},m_{2n+1}(r))-\lambda^{\mathcal{NN}}|<\frac{2\rho}{2n+1}\label{E:limitNN}.
\end{equation}
By adjusting $\{m_{2n+1}(r)\}_{n=1}^{+\infty}$, we can construct $\{m_{2n+2}(r)\}_{n=1}^{+\infty}\subset S_{\mathcal{NN}} $ such that
\begin{eqnarray}
|\lambda(s_{2n+1},m_{2n+2}(r))-\lambda^{\mathcal{DD}}|\le \frac{\rho}{2n+1},\label{nn1}
\end{eqnarray}
where $s_{2n+1}>\max\{s_{2n},8\mbox{/}\ln(1+\rho\mbox{/}(2n+2)c^*)\}$.
Then we obtain the sequence of functions $\{m_{n}(r)\}_{n=1}^{+\infty}$. Observe that $s_n\to+\infty$, $\tau_n\to 0$, $\delta(\tau_n)\to 0$ and $\kappa(\tau_n)\to \frac{1}{3}$ as $n\to+\infty$.

\textbf{Step 2:} (The construction of $\hat{m}(r)$)

Define $\hat{m}(r)=\lim_{n\to+\infty}m_n(r)$ and $m^*(r)=\lim_{n\to+\infty}m_n'(r)$. Since $m_n(r)\in C^1([0,1])$, to prove $\hat{m}(r)\in C^1([0,1])$, we only need to check that $\hat{m}'(r)=m^*(r)\in C([0,1])$. We will focus on the interval $[0,\frac{1}{3}]$ because $[\frac{2}{3},1]$ can be treated in the same way and $m_n(r)$ vanishes on $[\frac{1}{3},\frac{2}{3}]$ for all $n$.\par

 Noting that  $m_n'(r)$ is continuous and $\{m_n(r)\}_{n=1}^{+\infty}$ converges to $\hat{m}(r)$ pointwise according to the construction. It suffices to show that $\{m_n'(r)\}_{n=1}^{+\infty}$  converges uniformly to $m^*(r)$ on $[0,\frac{1}{3}]$. Since that $m_{n}'(r)$ is continuous on $[0,\frac{1}{3}]$ and $m_{n}'(\frac{1}{3})=0$, for any $\epsilon>0$, we can find a constant $\hat{\delta}>0$ such that
$$
|m_{n}'(r)-m_{n}'(a)|=|m_{n}'(r)|<\epsilon, \quad r\in\left(\frac{1}{3}-\hat{\delta},\frac{1}{3}\right].
$$
So we could taking $N_0$ large enough such that $\kappa(\tau_{N_0})>\frac{1}{3}-\hat{\delta}$. Then for any integers $n_1, n_2$ satisfying $n_2>n_1>N_0$, we have $m_{n_2}'(r)-m_{n_1}'(r)=0$ on $[0,\kappa(\tau_{n_1})]$ and
$$
|m_{n_2}'(r)-m_{n_1}'(r)|\le 2|m_{n_1}'(r)|<\epsilon, \quad r\in \left(\kappa(\tau_{n_1}),\frac{1}{3}\right].
$$
Hence, $m_n'(r)$ uniformly converge to $m^*(r)$ on $[0,\frac{1}{3}]$ and $\hat{m}'(r)=m^*(r)\in C([0,1])$.
\par

\textbf{Step 3:} (The non-existence of limit of $\hat{\lambda}(s)$ as $s\to+\infty$.)


By the definition of the sequence $\{m_n(r)\}_{n=1}^{+\infty}$ and $\hat{m}(r)$, for each $n$, we obtain
$$
|m_{2n+2}(s_{2n+1})-\hat{m}(r)|<2\tau_{2n+2}.
$$
Then, by  Lemma \ref{lemma}, we have
$$
|\lambda(s_{2n+1},m_{2n+2})-\lambda(s_{2n+1},\hat{m}(r))|\le c^*\left(e^{\frac{8}{s_{2n+2}}}-1\right)\le \frac{\rho}{2n+3}.
$$
and combining with \eqref{nn1} and utilizing  the triangle inequality,  we obtain
$$
|\lambda(s_{2n+1},\hat{m}(r))-\lambda^{\mathcal{DD}}|\le \frac{\rho}{n+1}+\frac{\rho}{2n+3}<\frac{2\rho}{n}.
$$
Thus, $\lim\limits_{n\to+\infty }\lambda(s_{2n+1},\hat{m}(r))= \lambda^{\mathcal{DD}} $. Similarly,  by \eqref{E:limitNN}, we could derive that $\lambda(s_{2n},\hat{m}(r))\to \lambda^{\mathcal{NN}} $ as $n\to +\infty$ and finish our proof.

\section*{Acknowledgement}
X. Bai was supported by NSF of China (Nos.12271437, 12071203). X. Xu was supported by Nankai Zhide Foundation. K. Zhang  was supported by Nankai Zhide Foundation.
M. Zhou was partially supported by National Key Research and Development Program of China (2021YFA1002400, 2020YFA0713300) and Nankai Zhide Foundation, NSF of China (Nos.12271437, 11971498), the Fundamental Research Funds For Central Universities (Nankai University 63231073) and Tianjin Natural Science Foundation Key Project (No.22JCZDJ0370).

\end{document}